\documentclass[11pt]{article} \usepackage{hyperref}
\usepackage{amsmath,amsthm,amssymb,stmaryrd,wasysym}
\usepackage[margin = 1in]{geometry} \usepackage{setspace}
\usepackage{enumitem} \usepackage{xcolor} \usepackage{tikz}
\usepackage{tikz-cd}
\theoremstyle{plain}
\newtheorem{theorem}{Theorem}[section]
\newtheorem{corollary}[theorem]{Corollary}
\newtheorem{lemma}[theorem]{Lemma}
\newtheorem{proposition}[theorem]{Proposition}
\theoremstyle{definition}
\newtheorem{definition}[theorem]{Definition}
\newtheorem{remark}[theorem]{Remark}
\newcommand{\biss}{{\sf BiS}s} \newcommand{\bis}{{\sf BiS}}
\newcommand{\viet}{{\mathcal V}} \newcommand{\im}{{\sf Im}}
\newcommand{\cl}{{\rm \it{Clop}}}
\newcommand{\cP}{\mathcal P} \newcommand{\cB}{\mathcal B}
\newcommand{\cC}{\mathcal C}

\newcommand{\just}[2]{\stackrel{#2}{#1}}
\begin{document}
\title{A note on powers of Boolean spaces with internal
  semigroups\footnote{This project has received funding from the
    European Research Council (ERC) under the European Union's Horizon
    2020 research and innovation program (grant agreement
    No. 670624).}}  \author{C\'elia Borlido, Mai Gehrke}\date{}
\maketitle

\begin{abstract}
  Boolean spaces with internal semigroups generalize profinite
  semigroups and are pertinent for the recognition of not-necessarily
  regular languages. Via recognition, the study of existential
  quantification in logic on words amounts to the study of certain
  spans of Boolean spaces with internal semigroups. In turn, these can
  be understood as the superposition of a span of Boolean spaces and a
  span of semigroups. In this note, we first study these
  separately. More precisely, we identify the conditions under which
  each of these spans gives rise to a morphism into the respective
  power or Vietoris construction of the corresponding
  structure. Combining these characterizations, we obtain such a
  characterization for spans of Boolean spaces with internal
  semigroups which we use to describe the topo-algebraic counterpart
  of monadic second-order existential quantification.  This is closely
  related to a part of the earlier work on existential quantification
  in first-order logic on words by Gehrke, Petri\c san and Reggio.
  The observation that certain morphisms lift contravariantly to the
  appropriate power structures makes our analysis very simple.
\end{abstract}
\section{Introduction}

The interaction of \emph{recognition} and \emph{quantification}
involves the analysis of spans of topological spaces of the form
\begin{equation}
  \begin{aligned}
    \begin{tikzcd}[node distance = 20mm]
      \node (S) at (0,0) {\beta S}; \node (T) [below left of = S]
      {\beta T}; \node (X) [below right of = S] {X};
      \draw[->] (S) to node[right, yshift = 2mm, xshift = -1mm] {g}
      (X); \draw[->] (S) to node[left, yshift = 2mm, xshift = 2mm]
      {\beta f} (T);
    \end{tikzcd}
  \end{aligned}
  \label{eq:7}
\end{equation}
where $T$ is the set of structures for the logic at hand, $S$ is the
set of all models over these structure (i.e. a set of free variables
is fixed and elements of $S$ consist of a structure from $T$ equipped
with an interpretation of the free variables). In the case of the
existential quantifier, the map $f: S \to T$ is just the map taking a
model to the underlying structure, and $\beta$ is the Stone-\v Cech
compactification (or, equivalently, $\beta f$ is the Stone dual of the
Boolean algebra homomorphism $f^{-1}: \cP(T) \to \cP(S)$). The idea of
recognition is that, instead of specifying a subalgebra of $\cP(S)$
corresponding to the formulas with free variables we want to study,
$g$ is a continuous map to a Boolean space dual to the subalgebra.

If $\cB$ is the Boolean algebra of clopens of $X$, then the Boolean
algebra $g^{-1}[\cB] = \{g^{-1}(U) \mid U \in \cB\}$ consists of the
model classes of the formulas with free variables that we want to
study. Also, for $L = g^{-1}(U) \cap S$, $L_\exists = f[L]$, is the
set of structures satisfying the corresponding (existentially)
quantified formula, and we are interested in building a recognizer for
the Boolean subalgebra of $\cP(T)$ generated by the sets
$f[g^{-1}(U)]$, for $U \in \cB$.

In~\cite{GehrkePetrisanReggio16}, it is shown that the desired
recognizer is of the form
\begin{equation}
  \label{eq:6}
  h: \beta T \to \viet(X)
\end{equation}
and is given dually by $\Diamond V \mapsto f[g^{-1}(V)]$, invoking
Vietoris for Boolean spaces as the dual of the monad for modal logic.
A categorical approach for a similar construction, also having in mind
applications in formal language theory, may be found
in~\cite{ChenUrbat16}.

A much simpler analysis would ensue if we can obtain $h$ by lifting
the pair~\eqref{eq:7} as follows
\begin{equation}
  \label{eq:8}
  \beta T \hookrightarrow \viet(\beta T) \xrightarrow{(\beta f)^*}
  \viet (\beta S) \xrightarrow{\viet(g)} \viet(X).
\end{equation}
Here $\viet(g)$ is the \emph{forward image} under $g$, while $(\beta
f)^*$ is the \emph{inverse image} under~$f$. It is well known that
forward image is always continuous and $\viet$ is indeed viewed as a
covariant endofunctor with $\viet(g)$ given by forward image. On the
other hand, inverse image under a continuous map is not in general
continuous on the Vietoris spaces.

Our short and simple analysis here consists in observing that for
certain continuous maps (those whose duals have lower adjoints, and
for the $\beta f$'s in particular) inverse image \emph{does} give a
continuous map. Indeed, we show that all continuous maps of the
form~\eqref{eq:6} come about from a span composing inverse and direct
image.

In the setting of recognition, we have to deal with a superposition of
semigroup structure and topology. We show that a similar phenomenon is
at play for semigroups and that these combine correctly to give a
simple description of the topo-algebraic recognizing maps of the
form~\eqref{eq:6} by means of a composition as in~\eqref{eq:8}.

The work in~\cite{GehrkePetrisanReggio16} and further papers in this
direction for quantification in \emph{first-order
  logic}~\cite{GehrkePetrisanReggio17, BorlidoCzarnetzkiGehrkeKrebs17}
deal with a second complication which is much deeper, and which stems
from the combined fact that the set $S$ of models does not carry an
appropriate semigroup structure and that $f$ is not a homomorphism of
semigroups.
We also mention that, although they do not refer to logic on words,
the corresponding treatment for the fragments of logic defining
\emph{regular languages} may be inferred from~\cite{AlmeidaWeil1995}.
Here, we treat a much simpler case, namely \emph{monadic second-order
  logic}, for which $T$ and $S$ are semigroups and $f$ is a
length-preserving semigroup homomorphism (i.e., one of those for which
inverse images lift homomorphically to the powerset semigroups).

This note is organized as follows.  In Section~\ref{sec:prelim}, we
set up the notation and define the main concepts used later.  In
Section~\ref{sec:powers} we study power constructions both for compact
Hausdorff and Boolean spaces and for semigroups. Finally, in
Section~\ref{sec:power-bis} we show that forward images under
length-preserving homomorphisms of languages recognized by a Boolean
space with an internal semigroup are precisely those that are
recognized by its Vietoris. This generalizes a known result for
regular languages~\cite{Reutenauer79,Straubing79}. Interpreting the result in the
context of logic on words leads to a topo-algebraic description of
second-order existential quantification.
\section{Preliminaries}\label{sec:prelim}
The reader is assumed to have some acquaintance with Stone duality and
with semigroups and the theory of formal languages, including logic on
words. Nevertheless, we briefly recall the main concepts and results
that we will need. For further reading on topology, we refer
to~\cite{Willard70}, on Stone duality to~\cite{DaveyPriestley02}, on
semigroup and formal language theory to~\cite{Almeida94}, and on logic
on words to~\cite{Straubing94}.

\paragraph{Stone duality.} \emph{Stone duality} establishes a
correspondence between Boolean algebras and Boolean spaces.\footnote{A
  Boolean space is a compact Hausdorff topological
  space with a basis of clopen sets.}
On the level of objects it goes as follows. Given a Boolean space~$X$,
the set of its clopen (i.e., both open and closed) subsets, $\cl(X)$,
forms a Boolean algebra (actually, this is true for every topological
space). Conversely, if~$\cB$ is a Boolean algebra, then the set of
ultrafilters\footnote{An ultrafilter is a proper maximal nonempty
  upset $x \subseteq \cB$ closed under binary meets.} of $\cB$, $X_\cB$, is
a Boolean space when equipped with the topology generated by the sets
of the form
\[\widehat {b} = \{x \in X_\cB \mid b \in x\},\]
for $b \in \cB$. We have isomorphisms $\cB \cong \cl(X_\cB)$ and $X \cong
X_{\cl(X)}$.
For the morphisms, we have the following correspondence. If $f: X \to
Y$ is a continuous function between Boolean spaces, then taking
preimages of clopens defines a homomorphism $f^{-1}:\cl(Y) \to \cl(X)$
of Boolean algebras, and if $\alpha: \cB \to \cC$ is a homomorphism of
Boolean algebras, for every ultrafilter $y \in X_\cC$, the set $\{b \in
\cB \mid \alpha (b) \in y\}$ is an ultrafilter of $\cB$, and this
correspondence yields a continuous function $X_\cC \to X_\cB$.

Dual spaces of powerset Boolean algebras will play a role in the
sequel. For any set~$S$, the Stone dual of $\cP(S)$ is the
\emph{Stone-\v Cech compactification of $S$}, which we will denote
by~$\beta S$. The set~$S$ embeds densely in~$\beta S$ via the map $s
\mapsto \{x \in \beta S \mid s \in x\}$.
We say that a subset $L \subseteq S$ is \emph{set-theoretically
  recognized} by a Boolean space~$X$ provided there is a continuous
map $f: \beta S \to X$ and a clopen subset $C \subseteq X$ such that
$L = f^{-1}(C) \cap S$.
Any function of sets $f: S \to T$ yields a homomorphism of Boolean
algebras $f^{-1}: \cP(T) \to \cP(S)$ and the corresponding dual map
will be denoted $\beta f: \beta S \to \beta T$.
\paragraph{Formal languages and recognition.} An \emph{alphabet} is a
finite set of symbols $A$, also called \emph{letters}, a (non-empty)
\emph{word} over $A$ is an element of the free semigroup $A^+$ and a
\emph{(formal) language} $L$ is a set of words over some alphabet.
The Boolean algebra of all languages over $A$, $\cP(A^+)$, is
naturally equipped with a biaction of $A^+$ given by
\begin{equation}
  u^{-1}Lv^{-1} = \{w \in A^+\mid uwv \in L\},\label{eq:1}
\end{equation}
for every $u,v \in A^+$ and $L \subseteq A^+$. Languages of the
form~\eqref{eq:1} are called \emph{quotients of $L$}.

We say that a language $L$ is \emph{recognized} by a semigroup $S$
provided there is a homomorphism $f:A^+ \to S$ and a subset $Q
\subseteq S$ satisfying $L = f^{-1}(Q)$, or equivalently, provided $L
= f^{-1}(f[L])$. Notice that the set of all languages recognized by a
given homomorphism forms a Boolean algebra closed under
quotients. Given a homomorphism $h: B^+ \to A^+$ between finitely
generated free semigroups and a language $L \subseteq A^+$ recognized
by~$S$, we have that $h^{-1}(L)$ is also recognized by~$S$. The
homomorphism~$h$ is said to be \emph{length-preserving}, or an
\emph{lp-morphism}, if it maps letters to letters. Languages
recognized by finite semigroups are called \emph{regular}.

To handle non-regular languages topo-algebraically, one needs to
consider richer structures, which we introduce below.
Notice that an attempt to make the notion of recognition as uniform as
possible was first made in~\cite{Bojanczyk15} and followed up
by~\cite{ChenAdamekMiliusUrbat2016}. But in both cases the main
concern is to capture several computational models within the same
framework, and not going beyond \emph{regularity}.
\paragraph{Boolean spaces with internal semigroups.}
This concept, for monoids, was introduced
in~\cite{GehrkePetrisanReggio16}, based on ideas
of~\cite{GehrkeGrigorieffPin10}. We give a short introduction. For
more details see~\cite{GehrkePetrisanReggio16}
or~\cite{GehrkePetrisanReggio17-arxiv}.

If $\cB \subseteq \cP(A^+)$ is a Boolean algebra of languages, then
dually, we have a continuous quotient $q: \beta (A^+)
\twoheadrightarrow X_\cB$. If the Boolean algebra~$\cB$ is closed
under quotients, then~$q[A^+]$ has a natural semigroup structure,
which is inherited from the duals of the homomorphisms
$u^{-1}(\_)v^{-1}: \cB \to \cB$. Moreover, the restriction of~$q$
to~$A^+$ induces a homomorphism $A^+ \twoheadrightarrow q[A^+]$ onto a
dense subset of~$X_\cB$. The pair $(q[A^+], X_\cB)$ encodes the
essential information about the Boolean algebra closed under
quotients~$\cB$. This is the reasoning motivating the definition of a
\emph{Boolean space with an internal semigroup}.

A \emph{Boolean space with an internal semigroup}, or \emph{\bis{}}
for short, is a pair $(S,X)$ where $S$ is a semigroup densely
contained in a Boolean space~$X$, and such that the natural actions of
$S$ on itself extend to continuous endomorphisms of~$X$, that is, for
each $s \in S$, there are continuous functions $\lambda_s, \rho_s: X
\to X$ such that the following diagrams commute:
\begin{center}
  \begin{tikzpicture}[node distance = 15mm]
    \node (S) at (0,0) {$S$};
    \node[below of = S] (SS) {$S$};
    \node[right of = S] (X) {$X$};
    \node[below of = X] (XX) {$X$};
    \node (Sr) at (6,0) {$S$}; \node[below of = Sr] (SSr) {$S$};
    \node[right of = Sr] (Xr) {$X$}; \node[below of = Xr] (XXr) {$X$};
    \draw[->] (S) to node[left] {\footnotesize $s\cdot (\_)$} (SS);
    \draw[>->] (S) to (X); \draw[>->] (SS) to (XX); \draw[->] (X) to
    node[right] {\footnotesize $\lambda_s$} (XX);
    \draw[->] (Sr) to node[left] {\footnotesize $(\_)\cdot s$} (SSr);
    \draw[>->] (Sr) to (Xr); \draw[>->] (SSr) to (XXr); \draw[->] (Xr)
    to node[right] {\footnotesize $\rho_s$} (XXr);
  \end{tikzpicture}
\end{center}
The prime example of a \bis{} is the pair $(A^+, \beta(A^+))$, which
corresponds to taking $\cB = \cP(A^+)$.

A morphism of \biss{}, $f: (S, X) \to (T, Y)$, is a continuous map
$f:X \to Y$ so that $f$ restricts to a semigroup homomorphism
$f{\upharpoonright} : S \to T$.

A language $L \subseteq A^+$ is recognized by the \bis{} $(S, X)$ if
there exists a morphism $f: (A^+, \beta(A^+)) \to (S, X)$ and a clopen
subset $C \subseteq X$ such that $L = f^{-1}(C) \cap A^+$. Notice that
each semigroup homomorphism $A^+ \to S$ completely determines a
morphism of \biss{} $(A^+, \beta(A^+)) \to (S, X)$. It is not hard to
verify that the set of all languages recognized by a \bis{} via a
fixed homomorphism is a Boolean algebra closed under quotients.
\paragraph{Logic on words.} As the name suggests, \emph{logic on
  words} is meant to express properties of words. We consider two
kinds of variables: first-order and second-order
variables. Intuitively, first-order variables provide information
about positions in a word, while the second-order variables stand for
sets of positions. First-order variables are denoted by $x, x_1, x_2,
\ldots$, and the second-order ones by $X, X_1, X_2, \ldots$.  There
are the following three types of atomic formulas:
\begin{itemize}
\item if $R \subseteq \mathbb N^k$, then $R(x_1, \ldots, x_k)$ is a
  (uniform) \emph{$k$-ary numerical predicate} (expressing that ``the
  tuple of positions $(x_1, \ldots, x_k)$ belongs to $R$'');
\item if $a \in A$, then ${\bf a}(x)$ is a \emph{letter predicate}
  (expressing that in the ``position $x$ there is an $a$'');
\item if $x$ is a first-order variable and $X$ is a second-order
  variable, then $X(x)$ is an atomic formula (expressing that ``$x$
  belongs to $X$'').
\end{itemize}
Then, Boolean combinations of formulas are formulas and if $\phi$ is a
formula, then both $\exists x \ \phi$ and $\exists X \ \phi$ are
formulas.
It is well known that every monadic second-order sentence is
equivalent, over words, to a formula of the form $\exists X_1 \cdots
\exists X_N \ \phi(X_1, \ldots, X_N)$, for some first-order formula
$\phi$ with free variables in $\{X_1, \ldots, X_N\}$.
To interpret formulas with $N$ second-order free variables, one
usually considers words over the extended alphabet $A \times 2^N$. For
instance, if $N = 1$, then the word $(a,1)(b,0)(b,1)(b,0)(a,1) \in (A
\times 2)^+$ encodes the word $w = abbba$ with the only second-order
variable interpreted in the set of odd positions of~$w$. Thus, every
formula $\phi = \phi (X_1, \ldots, X_N)$ with free variables in
$\{X_1, \ldots, X_N\}$ defines a language $L_{\phi(X_1, \ldots, X_N)}
\subseteq (A \times 2^N)^+$, and the language over $A^+$ definable by
$\exists X_1 \cdots \exists X_N \ \phi(X_1, \ldots, X_N)$ consists of
all the words $w$ for which there is an interpretation of the free
variables satisfying $\phi$. In other words, $T = A^+$ is the set of
all structures for logic on words in~$A$, $S = (A \times 2^N)^+$ is
the set of all MSO models in $N$ free variables on words in~$A$, and
the lp-morphism $\pi_N:(A \times 2^N)^+ \twoheadrightarrow A^+$, given
by the projection $A\times 2^N \twoheadrightarrow A$, gives rise to
existential quantification in the sense that $L_{\exists X_1 \cdots
  \exists X_N \ \phi(X_1, \ldots, X_N)} = \pi_N[L_{\phi (X_1, \ldots,
  X_N)}]$. More generally, given a language $L \subseteq (A \times
2^N)^+$, we denote $L_{\exists_N} = \pi_N[L]$.
\section{Some power constructions}\label{sec:powers}
\paragraph{Compact Hausdorff and Boolean spaces.}

The power of a compact Hausdorff spaces is the so-called
\emph{Vietoris space}.  \emph{Vietoris} is a covariant endofunctor on
compact Hausdorff spaces which restricts to the category of Boolean
spaces.
At the level of objects it assigns to a space~$X$ the set of all its
closed subsets, denoted $\viet(X)$ and called the \emph{Vietoris space
  of $X$}, equipped with the topology generated by the sets of the
form
\[\Diamond U = \{C \in \viet (X) \mid C \cap U \neq \emptyset\} \qquad
  \text{ and }\qquad \Box U = \{C \in \viet(X) \mid C \subseteq U\},\]
where $U \subseteq X$ ranges over all open subsets of~$X$.
In the case where $X$ is a Boolean space, taking $\Diamond U$ and
$\Box U$ for $U$ clopen gives a subbasis of clopen subsets
for~$\viet(X)$. For more details see~\cite{Michael51}.
Note that, since $X$ is Hausdorff, each singleton is closed and thus
the map $i_X\colon X\to\viet(X)$ sending each $x\in X$ to $\{x\}$ is
well defined. Further note that, for any open $U\subseteq X$, we have
\[
  \Diamond U\cap \im(i_X)=i_X[U] \qquad \text{and} \qquad \Box U \cap
  \im(i_X)=i_X[U]
\]
so that $X$ is homeomorphically embedded in $\viet(X)$ via the map $i_X$.
When the space $X$ is clear from the context, we denote this embedding
simply by~$i$. 

On morphisms, the Vietoris functor acts as follows: if $g: Z \to X$ is
a continuous function, then so is
\[\viet(g): \viet(Z) \to \viet(X), \qquad C \mapsto g[C].\]
Indeed, routine computations show that, for an open subset $U
\subseteq X$, we have
\begin{equation}
  \viet(g)^{-1}(\Diamond U) =
  \Diamond(g^{-1}(U)) \quad \text{ and }\quad \viet(g)^{-1}(\Box U) =
  \Box(g^{-1}(U)).\label{eq:2}
\end{equation}

On the other hand, given a continuous map $f: Z \to Y$, taking
preimages also defines a function between the corresponding Vietoris
spaces, but in a contravariant way. A natural question is then under
which conditions the function $f^* := f^{-1}: \viet(Y) \to \viet(Z)$
is continuous.

\begin{proposition}\label{p:3}
  Let $f: Z \to Y$ be a continuous function between compact Hausdorff
  spaces. Then, the following are equivalent:
  \begin{enumerate}[label = (\alph*)]
  \item\label{item:1} $f^*: \viet(Y) \to \viet(Z)$ is continuous,
  \item\label{item:2} $f^* \circ i: Y \to \viet(Z)$ is continuous,
  \item\label{item:3} $f$ is open.
  \end{enumerate}
\end{proposition}
\begin{proof}
  Since $Y$ is homeomorphically embedded in $Y$ via the map $i$, it is
  clear that $\ref{item:1}$ implies $\ref{item:2}$.
The remainder of the proposition is essentially a consequence of the fact that,
for any set map $f$, the forward image  under $f$ is lower adjoint to the inverse
image under $f$. That is,
\[
  \forall\ S\subseteq Y,\ T\subseteq Z\qquad \qquad(\ f[T]\subseteq S
  \ \iff\ T\subseteq f^{-1}(S)\ ).
\]
Using this, we see first of all that, for compact Hausdorff spaces, the map $f^*$ is always 
continuous with respect to the `box-part' of the topology. That is, for any open $U\subseteq Z$ 
and closed $K\subseteq Y$, we have 
\begin{align*}
  K\in (f^*)^{-1}(\Box U) \iff f^{-1}(K)\subseteq U
  &\iff U^c\subseteq(f^{-1}(K))^c=f^{-1}(K^c)\\
  &\iff f[U^c]\subseteq K^c\iff K\subseteq(f[U^c])^c.
\end{align*}
Now, since $f$ is continuous, $Y$ is compact, and $Z$ is Hausdorff, it follows that $f$ is a 
closed mapping and thus $(f[U^c])^c$ is open. That is, we have shown that 
\[
(f^*)^{-1}(\Box U) = \Box (f[U^c])^c.
 \]
 Similarly, we have
 \begin{align*}
   K\in (f^*)^{-1}(\Diamond U) \iff f^{-1}(K)\cap U\neq\emptyset
   &\iff U\not\subseteq(f^{-1}(K))^c=f^{-1}(K^c)\\
   &\iff f[U]\not\subseteq K^c\iff K\cap f[U]\neq\emptyset.
 \end{align*}
 Now, if $f$ is open, then the above calculation shows
 $(f^*)^{-1}(\Diamond U) = \Diamond f[U]$ and thus that $\ref{item:1}$
 and $\ref{item:2}$ hold. Conversely, if $f^*\circ i$ is continuous,
 then, for any open $U\subseteq Z$, the set
 \begin{equation}
   (f^*\circ i)^{-1}(\Diamond U)=\{y\in Y\mid \{y\}\cap
   f[U]\neq\emptyset\}=f[U]\label{eq:4}
 \end{equation}
 must be open.
\end{proof}

In particular we have the following:
\begin{corollary}\label{c:2}
  Let $X$, $Y$ and $Z$ be compact Hausdorff spaces and $f: Z \to Y$
  and $g: Z \to X$ be continuous functions. If $f$ is an open map,
  then
  \[h = \viet(g) \circ f^* \circ i: Y \to \viet(X), \qquad y \mapsto
    h(y) = g[f^{-1}(\{y\})]\]
  is continuous. Moreover, for every open subset $U \subseteq X$, the
  equality $h^{-1}(\Diamond U) = f[g^{-1}(U)]$ holds.
\end{corollary}
\begin{proof}
  The fact that $h$ is continuous follows immediately from
  Proposition~\ref{p:3}. Moreover, for an open subset $U \subseteq X$,
  we have:
  \[h^{-1}(\Diamond U) = (f^*\circ i)^{-1}(\viet(g)^{-1}(\Diamond U))
    \just = {\eqref{eq:2}} (f^* \circ i)^{-1}(\Diamond g^{-1}(U))
    \just = {\eqref{eq:4}} f[g^{-1}(U)]. \popQED\]  
\end{proof}

In the next proposition we show that every continuous map $Y \to
\viet(X)$ arises from a composition as in Corollary~\ref{c:2}.

\begin{proposition}\label{p:5}
  Let $X$ and $Y$ be compact Hausdorff spaces and $h: Y \to \viet(X)$
  be a continuous function. Then, the subspace
  \[Z = \{(y, x) \in Y \times X \mid x \in h(y)\}\]
  of $Y \times X$ is compact and Hausdorff. In particular, $h$ is of
  the form $\viet(g)\circ f^* \circ i$, where $f$ and $g$ are the
  restrictions to $Z$ of the projections to~$Y$ and~$X$, respectively.
\end{proposition}
\begin{proof}
  The space $Z$ is compact Hausdorff if $Z$ is a closed subspace of~$Y
  \times X$.
  We show that $Z^c$ is open. Let $(y,x) \notin Z$. Then, $x \notin
  h(y)$ and since $h(y) \subseteq X$ is closed and $X$ is compact and
  Hausdorff (and thus, regular), there exist disjoint open subsets $U,
  V \subseteq X$ so that $x \in U$ and $h(y) \subseteq V$. Since $h$
  is continuous, it follows that $h^{-1}(\Box V) \times U$ is an open
  neighborhood of $(y,x)$ contained in $Z^c$.
\end{proof}

We finish this section with a characterization of the open maps
between Boolean spaces. Recall that, in general, a \emph{lower
  adjoint} of a map $\alpha: P \to Q$ of posets is a map $\alpha_*: Q
\to P$ satisfying
\[
  \forall\ p \in P,\ q \in Q \qquad \qquad(\ \alpha_*(q)\le p \ \iff\
  q \le \alpha(p)\ ).
\]

\begin{proposition}
  Let $\cB$ and $\cC$ be Boolean algebras, with duals $Y$ and $Z$,
  respectively. Then, a continuous function $f: Z \to Y$ is open if
  and only if the dual map $\alpha: \cB \to \cC$ has a lower adjoint.
\end{proposition}
\begin{proof}
  First observe that, since taking forward images preserves arbitrary
  unions, $f$ is open if and only if $f[\,\widehat c\, ]$ is open for
  every $c \in \cC$. Since $f[\, \widehat c\ ]$ is closed, we have
  \[ f[\, \widehat c \ ] = \bigcap \{\widehat b \mid b \in \cB, \
    f[\,\widehat c \ ] \subseteq \widehat b\}.\]
  Thus, $f$ is open exactly when the set
  \begin{equation}
    \{\widehat b \mid b \in \cB,
    \ f[\,\widehat c \ ] \subseteq \widehat b\}\label{eq:5}
  \end{equation}
  has a minimum for inclusion. On the other hand, the fact that
  $\alpha$ and $f$ are dual to each other, translates to the fact
  that, for every $b \in \cB$ and $c \in \cC$, we have
  \[ f[\, {\widehat c}\ ] \subseteq \widehat b \iff \widehat c
    \subseteq f^{-1}(\widehat b) = \widehat{\alpha(b)} \iff c \le
    \alpha(b).\]
  Therefore, \eqref{eq:5} has a minimum if and only if $\{b \in \cB
  \mid c \le \alpha(b)\}$ does. But this amounts to saying that
  $\alpha$ admits a lower adjoint.
\end{proof}

In particular, since every function $f: S \to T$ between sets $S$ and
$T$ induces a complete homomorphism $f^{-1}: \cP(T) \to \cP(S)$
between complete Boolean algebras, thus having a lower adjoint, we
have the following:
\begin{corollary}\label{c:1}
  For every map of sets $f: S \to T$, the map $\beta f: \beta S \to
  \beta T$ is open.
\end{corollary}

\begin{remark}\label{r:2}
  We remark that, so far, we proved that, given a set map $f: S \to T$
  and a continuous function $g: \beta T \to X$ the Boolean algebra
  generated by the subsets of the form $f[g^{-1}(U)]$, for $U \in
  \cl(X)$, is precisely $h^{-1}[\cl(\viet(X))] = \{h^{-1}(V) \mid V
  \in \cl(\viet(X))\}$, where $h = \viet(g) \circ f^* \circ i$.
\end{remark}
\paragraph{Semigroups.} We start by recalling that, given a semigroup
$S$, the set $\cP(S)$ of its subsets is equipped with a semigroup
structure given by pointwise multiplication:
$$Q_1 \cdot Q_2 = \{s_1s_2 \mid s_1\in Q_1, s_2 \in Q_2\},$$
for every subsets $Q_1,Q_2 \subseteq S$. In particular, there is an
embedding of semigroups $i_S: S \hookrightarrow \cP(S)$ given by
$i_S(s) = \{s\}$. When $S$ is clear from the context, we just
write~$i$. Notice that $\cP(S)$ also admits a monoid structure, with
the neutral element being the empty set, but we are only concerned
with the semigroup structure of~$\cP(S)$. Taking powers defines an
endofunctor on semigroups. Indeed, for a homomorphism $g: S \to T$,
taking forward images defines a homomorphism
\[\cP(g): \cP(S) \to \cP(T), \qquad Q \mapsto g[Q].\]
Of course, the set $\cP_{fin}(S)$ consisting of the finite subsets of
$S$ forms a subsemigroup of $\cP(S)$, and for a homomorphism $g: S \to
T$, $\cP(g)$ restricts to a homomorphism $\cP_{fin}(g): \cP_{fin}(S)
\to \cP_{fin}(T)$. This observation will be useful in
Section~\ref{sec:powers}.

On the other hand, if $f: S \to T$ is a semigroup homomorphism, then
taking preimages defines a map between the corresponding powersets,
which in general is not a homomorphism. However, we do have the next
result.
\begin{lemma}\label{l:1}
  Let $f: B^+ \to A^+$ be a homomorphism between free
  semigroups. Then, the following are equivalent:
  \begin{enumerate}[label = (\alph*)]
  \item\label{item:4} $f^{*}: \cP(A^+) \to \cP(B^+)$ is a
    homomorphism,
  \item\label{item:5} $f^*\circ i: A^+ \to \cP(B^+)$ is a
    homomorphism,
  \item\label{item:6} $f$ is an lp-morphism.
  \end{enumerate}
\end{lemma}
\begin{proof}
  The equivalence between~\ref{item:4} and~\ref{item:5} is
  trivial. Suppose that $f$ is length-preserving. Then, given $w_1,
  w_2 \in A^+$ and $u \in B^+$, we have that $u \in f^*\circ
  i(w_1w_2)$ if and only if $f(u) = w_1w_2$. Since $f$ is
  length-preserving, this happens if and only if $u$ admits a
  factorization $u = u_1u_2$ satisfying $f(u_1) = w_1$ and $f(u_2) =
  w_2$, that is, $u \in f^*\circ i(w_1)\cdot f^*\circ i(w_2)$. This
  proves that $f^*\circ i$ is a homomorphism. Conversely, if $f$ is
  not length-preserving, then there exists a letter $b \in B$ so that
  $f(b)$ may be written as $aw$ for some $a \in A$ and $w \in
  A^+$. Then, $b$ belongs to $f^*\circ i(aw)$ but not to $f^*\circ
  i(a)\cdot f^*\circ i(w)$ and so, $f^*\circ i$ is not a homomorphism.
\end{proof}

\begin{remark}\label{r:1}
  Notice that for an lp-morphism $f: B^+ \to A^+$ and a word $w \in
  A^+$, the set $f^{-1}(w)$ is finite. Thus, by Lemma~\ref{l:1}, every
  such~$f$ defines a homomorphism of semigroups $f^*: A^+ \to
  \cP_{fin}(B^+)$.
\end{remark}

The following is a particular case of a well known result in semigroup
theory (see e.g.~\cite[Chapter~XVI, Proposition~1.1]{Pin-notes}).
\begin{proposition}\label{p:4}
  Let $h: T \to \cP(S)$ be a homomorphism. Then, the set
  \[R = \{(t, s) \mid t \in T, \ s \in h(t)\}\]
  is a subsemigroup of $T \times S$. In particular, $h$ is of the form
  $\cP(g) \circ f^* \circ i$, where $f$ and $g$ are the restrictions
  to~$R$ of the projections to~$S$ and~$T$, respectively.
\end{proposition}
\section{The power construction for \biss{} and
  MSO}\label{sec:power-bis}

\paragraph{The power construction for \biss{}.}

Combining the power constructions of Section~\ref{sec:powers} provides
a power construction that applies to \biss{}:

\begin{definition}[{\cite[Theorem III.1]{GehrkePetrisanReggio17}}]
  Let $(S, X)$ be a \bis{}. We define the \emph{Vietoris of $(S,X)$}
  to be the~\bis{}
  \[\viet(S,X) = (\cP_{fin}(S), \viet(X))\]
  equipped with the actions
  \[\lambda_Q: \viet(X) \to \viet(X), \quad C \mapsto \bigcup_{s
      \in Q} \lambda_s[C] \qquad\text{ and }\qquad \rho_Q: \viet(X)
    \to \viet(X), \quad C \mapsto \bigcup_{s \in Q} \rho_s[C],\]
  for each $Q \in \cP_{fin}(S)$.
\end{definition}

We remark that, although the fact that $\viet(S, X)$ is a \bis{} only
appears explicitly in~\cite{GehrkePetrisanReggio17}, this is
implicitly present already in~\cite{GehrkePetrisanReggio16}.

We will say that a morphism $h: (B^+, \beta(B^+)) \to (A^+,
\beta(A^+))$ of \biss{} is \emph{length-preserving} provided its
restriction to $B^+$ is length-preserving. Using Corollary~\ref{c:2}
and Lemma~\ref{l:1}, and taking into account Remark~\ref{r:1}, we
obtain:
\begin{proposition}
  Let $A$ and $B$ be alphabets and $(S, X)$ a \bis{}. Then, for every
  span
  \begin{equation*}
    \begin{aligned}
      \begin{tikzcd}[node distance = 25mm]
        \node (S) at (0,0) {(B^+, \beta (B^+))}; \node (T) [below left
        of = S] {(A^+, \beta(A^+))}; \node (X) [below right of = S]
        {(S, X)};
        \draw[->] (S) to node[right, yshift = 2mm, xshift = -1mm] {g}
        (X); \draw[->] (S) to node[left, yshift = 2mm, xshift = 2mm]
        {f} (T);
      \end{tikzcd}
  \end{aligned}
\end{equation*}
with $f$ length preserving, the map $h = \viet(g) \circ f^* \circ i$
is a morphism of \biss{}.
\end{proposition}
\begin{proposition}\label{p:9}
  Let $(S, X)$ and $(T, Y)$ be \biss{} and $h: (T, Y) \to \viet(S, X)$
  a morphism. Then, there is a \bis{} $(R, Z)$ and morphisms $f:(R, Z)
  \to (T, Y)$ and $g: (R, Z) \to (S, X)$ so that $h = \viet(g) \circ
  f^* \circ i$.
\end{proposition}
\begin{proof}
  We take $Z = \{(y,x) \mid y \in Y, \ x \in h(y)\}$ and $R = \{(t,s)
  \mid t \in T, \ s \in h(s)\}$.  By Propositions~\ref{p:5}
  and~\ref{p:4} we already know that $Z$ and $R$ are, respectively, a
  Boolean space and a semigroup that do the job. Thus, it remains to
  show that $(R, Z)$ is a \bis{}. Since the pair $(T \times S, Y
  \times X)$ has a \bis{} structure, we only need to prove that $R$ is
  dense in $Z$. Let $V\subseteq Y$ and $U \subseteq X$ be open subsets
  and $(y, x) \in (V \times U) \cap Z$. We need to show that $(V
  \times U) \cap R$ is nonempty. Since $h$ is continuous,
  $h^{-1}(\Diamond U) \cap V$ is an open subset of~$Y$, and it is
  nonempty as it contains~$(y,x)$. Since $T$ is dense in~$Y$, there
  exists an element $t \in h^{-1}(\Diamond U) \cap V \cap T$. In
  particular, $h(t) \cap U \neq \emptyset$. Since $h$ restricts to a
  semigroup homomorphism $T \to \cP_{fin}(S)$, this yields the
  existence of $s \in h(t) \cap U \cap S$ as required.
\end{proof}

As a consequence we obtain the desired result on recognition.

\begin{corollary}\label{c:3}
  Let $(S, X)$ be a \bis{}. Then, a language $L \subseteq A^+$ is
  recognized by $\viet(S, X)$ if and only if it is a Boolean
  combination of forward images under lp-morphisms of languages
  recognized by~$(S, X)$.
\end{corollary}
\begin{proof}
  The backwards implication is a trivial consequence of
  Remark~\ref{r:2} and Proposition~\ref{p:9}.
  Conversely, let $h: (A^+, \beta(A^+)) \to \viet(S, X)$ be a morphism
  recognizing $L \subseteq A^+$. Again by Proposition~\ref{p:9}, there
  exists a \bis{} $(R,Z)$ and morphisms $f:(R, Z) \to (A^+,
  \beta(A^+))$ and $g: (R, Z) \to (S, X)$ so that $h = \viet(g) \circ
  f^* \circ i$. The only thing to notice is that $R = \{(u, s) \mid u
  \in A^+, s \in h(u)\}$ is the subsemigroup of $A^+ \times S$
  generated by the finite alphabet $B = \{(a, s) \mid a \in A, \ s \in
  h(a)\}$. Therefore, we have a unique morphism of \biss{} $\pi: (B^+,
  \beta(B^+)) \to (R, Z)$ mapping $b \in B$ to $b \in R$, and this
  morphism is such that $f \circ \pi$ is length-preserving. The
  intended conclusion follows then from Remark~\ref{r:2}.
\end{proof}

We remark that this result takes care of the first stage of
set-theoretic recognition in the first-order
setting~\cite{GehrkePetrisanReggio16,
  BorlidoCzarnetzkiGehrkeKrebs17}. The complication in the first-order
setting stems from the fact that the first-order models do not form a
semigroup. Here we treat the considerable easier case of monadic
second-order quantifiers.
\paragraph{Monadic second-order existential quantification.}

We address the following questions: Given a \bis{} $(S, X)$, which
\bis{} recognizes the Boolean algebra generated by the languages of
the form $L_{\exists N}$ where $L$ is recognized by $(S, X)$?  Does it
recognize much more?  Unlike in the first-order case where an
iterative construction is absolutely needed, we will see that, for
second-order quantification, taking once the power of $(S, X)$ is
enough to recognize every language $L_{\exists N}$ for every $N \in
\mathbb N$, where $L$ is recognized by~$(S, X)$. In fact, since the
projection $\pi_N: (A\times 2^N)^+ \twoheadrightarrow A^+$ modeling
the quantifier $\exists_N$ is length-preserving, this essentially
follows from the results above. As already mentioned, the
corresponding problem for first-order quantifiers was considered
in~\cite{GehrkePetrisanReggio16, BorlidoCzarnetzkiGehrkeKrebs17} and
it is much more delicate, as the universe of models of formulas with
free first-order variables does not admit a semigroup structure.

\begin{proposition}\label{p:8}
  Let $(S, X)$ be a \bis{}. If $(S, X)$ recognizes the language $L
  \subseteq (A \times 2^N)^+$, then $\viet(S, X)$ recognizes the
  language $L_{\exists_N} \subseteq A^+$. Conversely, if $K \subseteq
  A^+$ if recognized by $\viet(S, X)$, then there exists a positive
  integer $N$, an alphabet $A' \subseteq A$ and a language $L
  \subseteq (A' \times 2^N)^+$ recognized by $(S, X)$ such that $K =
  L_{\exists_N}$.
\end{proposition}
\begin{proof}
  Since $L_{\exists_N} = \pi_N[L]$ and $\pi_N$ is an lp-morphism, the
  first part is a consequence of Corollary~\ref{c:3}.
  Conversely, let $K \subseteq A^+$ be a language recognized by
  $\viet(S, X)$. By the proof of Corollary~\ref{c:3}, there is a
  finite alphabet $B \subseteq A\times S$ such that $K = f[L]$ for a
  language $L \subseteq B^+$ recognized by $(S, X)$, where $f: B \to
  A$ is the restriction of the projection $A \times S \to A$. Take $A'
  = f[B]$ and choose a big enough~$N$ so that each of the sets $B \cap
  (\{a\} \times S)$, with $a \in A'$, has at most~$N$ elements. Then,
  there exists an onto homomorphism $\pi: (A' \times 2^N)^+
  \twoheadrightarrow B^+$ satisfying $\pi_N = f \circ \pi$, and so,
  there is a language $L' = \pi^{-1}(L) \subseteq (A' \times 2^N)^+$
  recognized by $(S, X)$ and such that $K = L'_{\exists_N}$.
\end{proof}

\begin{remark}
  Observe that, if $B' \subseteq B$ is an inclusion of alphabets, then
  a language $L \subseteq (B')^+$ can always be seen as a language
  over $B$ via the inclusion $(B')^+ \subseteq B^+$. Nevertheless, the
  fact that $L$ is recognized by $(S, X)$ as a language over $B'$ does
  not necessarily implies that it is recognized by $(S, X)$ as a
  language over $B$. On the other hand, the \bis{} $(S \times {\bf 2},
  X \times \{0,1\})$, where ${\bf 2}$ denotes the two-element
  semilattice, does recognize $L$ as a language over~$B$. Since one is
  usually interested in studying fragments of logic and not a single
  formula, we can always express the property ``every letter belongs
  to $B'$'', and so, this is not really a constraint.
\end{remark}

Notice that, Proposition~\ref{p:8} implies that a language is
recognized by the power of an aperiodic semigroup if and only if it is
obtained by second-order quantification of a language recognized by
aperiodic semigroups. On the other hand, the languages recognizable by
an aperiodic semigroup are precisely those definable in ${\bf
  FO}[<]$~\cite{Schutzenberger65,McNaughtonPapert71}, and in turn, the
second-order existential quantification of those yields ${\bf
  MSO}[<]$, a fragment of logic defining precisely the regular
languages~\cite{Buchi60,Elgot61} (i.e., those recognized by a finite
semigroup). We may thus derive the following:

\begin{corollary}
  Every finite semigroup divides a power of an aperiodic one.
\end{corollary}

\footnotesize
\bibliographystyle{plain}

\begin{thebibliography}{10}

\bibitem{Almeida94}
J.~Almeida.
\newblock {\em Finite semigroups and universal algebra}.
\newblock World Scientific Publishing Co. Inc., River Edge, NJ, 1994.
\newblock Translated from the 1992 Portuguese original and revised by the
  author.

\bibitem{AlmeidaWeil1995} J.~{Almeida} and P.~{Weil}.  \newblock {Free
    profinite semigroups over semidirect products.}  \newblock {\em
    {Russian Math. (Iz. VUZ)}}, 39(1):1--27, 1995.

\bibitem{Bojanczyk15}
M.~Boja\'nczyk.
\newblock Recognisable languages over monads.
\newblock In {\em Developments in language theory}, volume 9168 of {\em Lecture
  Notes in Comput. Sci.}, pages 1--13. Springer, Cham, 2015.

\bibitem{BorlidoCzarnetzkiGehrkeKrebs17}
C.~Borlido, S.~Czarnetzki, M.~Gehrke, and A.~Krebs.
\newblock Stone duality and the substitution principle.
\newblock In {\em Computer science logic 2017}, volume~82 of {\em LIPIcs.
  Leibniz Int. Proc. Inform.}, pages Art. No. 13, 20. Schloss Dagstuhl.
  Leibniz-Zent. Inform., Wadern, 2017.

\bibitem{Buchi60}
J.~R. B\"uchi.
\newblock Weak second-order arithmetic and finite automata.
\newblock {\em Z. Math. Logik Grundlagen Math.}, 6:66--92, 1960.

\bibitem{ChenAdamekMiliusUrbat2016}
L.-T. Chen, J.~Ad\'{a}mek, S.~Milius, and H.~Urbat.
\newblock Profinite monads, profinite equations, and {R}eiterman's theorem.
\newblock In {\em Foundations of software science and computation structures},
  volume 9634 of {\em Lecture Notes in Comput. Sci.}, pages 531--547. Springer,
  Berlin, 2016.

\bibitem{ChenUrbat16}
L.-T. Chen and H.~Urbat.
\newblock Sch\"utzenberger products in a category.
\newblock In {\em Developments in language theory}, volume 9840 of {\em Lecture
  Notes in Comput. Sci.}, pages 89--101. Springer, Berlin, 2016.


\bibitem{DaveyPriestley02}
B.~A. Davey and H.~A. Priestley.
\newblock {\em Introduction to Lattices and Order, 2nd edition}.
\newblock Cambridge University Press, 2002.

\bibitem{Elgot61}
C.~C. Elgot.
\newblock Decision problems of finite automata design and related arithmetics.
\newblock {\em Trans. Amer. Math. Soc.}, 98:21--51, 1961.

\bibitem{GehrkeGrigorieffPin10}
M.~Gehrke, S.~Grigorieff, and J.-\'{E}. Pin.
\newblock A topological approach to recognition.
\newblock In {\em Automata, languages and programming. {P}art {II}}, volume
  6199 of {\em Lecture Notes in Comput. Sci.}, pages 151--162. Springer,
  Berlin, 2010.

\bibitem{GehrkePetrisanReggio16}
M.~Gehrke, D.~Petri\c san, and L.~Reggio.
\newblock The {S}ch\"utzenberger product for syntactic spaces.
\newblock In {\em 43rd {I}nternational {C}olloquium on {A}utomata, {L}anguages,
  and {P}rogramming}, volume~55 of {\em LIPIcs. Leibniz Int. Proc. Inform.},
  pages Art. No. 112, 14. Schloss Dagstuhl. Leibniz-Zent. Inform., Wadern,
  2016.
  
\bibitem{GehrkePetrisanReggio17}
M.~Gehrke, D.~Petri\c{s}an, and L.~Reggio.
\newblock Quantifiers on languages and codensity monads.
\newblock In {\em 2017 32nd {A}nnual {ACM}/{IEEE} {S}ymposium on {L}ogic in
  {C}omputer {S}cience ({LICS})}, page~12. IEEE, [Piscataway], NJ, 2017.

\bibitem{GehrkePetrisanReggio17-arxiv}
M. Gehrke, D. Petri\c san, and L. Reggio.
\newblock Quantifiers on languages and codensity monads.
\newblock {\em CoRR}, abs/1702.08841, 2017.

\bibitem{McNaughtonPapert71}
R.~McNaughton and S.~Papert.
\newblock {\em Counter-free automata}.
\newblock The M.I.T. Press, Cambridge, Mass.-London, 1971.
\newblock With an appendix by William Henneman, M.I.T. Research Monograph, No.
  65.

\bibitem{Michael51}
E.~Michael.
\newblock Topologies on spaces of subsets.
\newblock {\em Trans. Amer. Math. Soc.}, 71:152--182, 1951.

\bibitem{Pin-notes}
J.-\'E. Pin.
\newblock Mathematical foundations of automata theory.
\newblock Course notes, version of June 15, 2018.

\bibitem{Reutenauer79}
C.~Reutenauer.
\newblock Sur les vari\'et\'es de langages et de mono\"\i des.
\newblock In {\em Theoretical computer science ({F}ourth {GI} {C}onf.,
  {A}achen, 1979)}, volume~67 of {\em Lecture Notes in Comput. Sci.}, pages
  260--265. Springer, Berlin-New York, 1979.

\bibitem{Schutzenberger65}
M.~P. Sch\"utzenberger.
\newblock On finite monoids having only trivial subgroups.
\newblock {\em Information and Control}, 8:190--194, 1965.

\bibitem{Straubing79}
H.~Straubing.
\newblock Recognizable sets and power sets of finite semigroups.
\newblock {\em Semigroup Forum}, 18(4):331--340, 1979.

\bibitem{Straubing94}
H.~Straubing.
\newblock {\em Finite automata, formal logic, and circuit complexity}.
\newblock Progress in Theoretical Computer Science. Birkh\"auser Boston, Inc.,
  Boston, MA, 1994.

\bibitem{Willard70}
S.~Willard.
\newblock {\em General topology}.
\newblock Dover Publications, Inc., Mineola, NY, 2004.
\newblock Reprint of the 1970 original.

\end{thebibliography}

\end{document}